\documentclass[preprint,12pt]{amsart}
\usepackage{amsmath}
\usepackage{amsfonts}
\usepackage{amssymb}
\usepackage{amsthm}
\usepackage{epsfig}
\usepackage{graphicx}
\usepackage{url,hyperref}
\usepackage{enumerate}

\def\s{\mathbb{S}}

\def\R{\mathbb{R}}
\def\N{\mathbb{N}}

\def\vol{\mathrm{vol}}

\theoremstyle{remark}
\newtheorem{theorem}{Theorem}
\newtheorem{remark}[theorem]{Remark}

\newtheorem{lemma}{Lemma}
\newtheorem{corollary}{Corollary}

\theoremstyle{definition}
\newtheorem{definition}[theorem]{Definition}
\newtheorem{example}[theorem]{Example}

\begin{document}

\title{Rogers-Shephard Type Inequalities for Sections}
\author[M. Roysdon]{Michael Roysdon}
\address{Department of Mathematical Sciences, Kent State University, Kent,
OH USA} \email{mroysdon@kent.edu}

\begin{abstract}

	In this paper we address the following question: given a measure $\mu$ on $\R^n$, does there exists a constant $C>0$ such that, for any $m$-dimensional subspace $H \subset \R^n$ and any convex body $K \subset \R^n$, the following sectional Rogers-Shephard type inequality holds:
	\[
	\mu((K-K) \cap H) \leq C \sup_{y \in \R^n} \mu(K \cap (H+y))?
	\]
	We show that this inequality is affirmative in the class of measures with radially decreasing densities with the constant $C(n,m) = \binom{n+m}{m}$. We also prove marginal inequalities of the Rogers-Shephard type for $\left(\frac{1}{s}\right)$-concave, $0 \leq s < \infty$, and logarithmically concave functions.

\end{abstract}

\maketitle
\footnote{Research partially supported by Erasmus+ grant for the 2018/2019 academic year}

\section{Introduction and main results}

By $(\R^n, |\cdot|)$ we denote the $n$-dimensional real Euclidean space with its usual metric structure. A {\it convex body} is a compact convex subset of $\R^n$ with non-empty interior. We will say that a convex body $K \subset \R^n$ is {\it symmetric} if, for some $x \in \R^n$, $K-x=-(K-x)$. 
We represent by $B_n$ the
$n$-dimensional Euclidean unit ball, and by $\s^{n-1}$ its boundary. The $n$-dimensional volume of a measurable set
$M\subset\R^n$, i.e., its $n$-dimensional Lebesgue measure, is
denoted by $\vol_n(M)$. Moreover, we denote by $G_{n,m}$ the set of $m$-dimensional linear subspaces of $\R^n$, and given $H \in G_{n,m}$, we shall denote by $H^{\perp}$ the orthogonal complement of $H$. For a set $A \subset \R^n$, let $\chi_A$ denote the characteristic function of $A$. 

The {\it Minkowski addition} of two sets $A,B \subset \R^n$ is defined by their usual vector sum:
\[
A+B =\{a+b \colon a \in A, b \in B\},
\]
and we shall write $A-B$  for $A+(-B)$. 

Connecting the Minkowski addition of convex bodies to their volume is the famed Brunn-Minkowski inequality, one form of which may be stated as follows: given convex bodies $K,L \subset \R^n$, then 
\begin{equation*}
\vol_n(K+L)^{\frac{1}{n}} \geq \vol_n(K)^{\frac{1}{n}} + \vol_n(L)^{\frac{1}{n}}
\end{equation*}
with equality if and only if $K$ and $L$ are homothetic (see \cite{G} for an extensive survey on the Brunn Minkowski inequality). In particular, in the case when $L=-K$, one has $\vol_n(K-K) \geq 2^n \vol_n(K)$ with equality only when $K$ is symmetric. A reverse inequality of this was discovered by Rogers and Shephard in the 1950s, the so-called Rogers-Shephard inequality (see
\cite[Theorem~1]{RS1} and \cite[Section~10.1]{Sch}).
The Rogers-Shephard inequality reads: given any convex body $K \subset \R^n$,
\begin{equation}\label{e:RS}
\vol_n(K-K)\leq \binom{2n}{n}\vol_n(K),
\end{equation}
with equality if and only if $K$ is an $n$-dimensional simplex.

Alternatively, one can view the Minkowski sum in the following way 
\begin{align*}
K+L = \{x \in \R^n \colon K \cap (x-L) \neq \emptyset\}. 
\end{align*}
With this interpretation of the Minkowski sum of sets, given any convex bodies $K$ and $L$, the Brunn-Minkowski inequality implies that the function 
\[
f(x) = \vol_n(K \cap (x-L))^{\frac{1}{n}}
\]
is concave on $K+L$.

In recent years, both the Brunn-Minkowski inequality and the Rogers-Shephard inequalities have been studied deeply and extended to larger classes of measures on $\R^n$.  For results on the Brunn-Minkowski inequality see \cite{Borell,BL,G,GaZv,Leindler,LiMaNaZv,Mar,NaTk,Prekopa,RYN}, and for generalizations of the Rogers-Shephard inequality see \cite{A,AAGJV,ACRYZ,AlGMJV,Co,Sch}.

One of the most famous extensions of the Brunn-Minkowski inequality is the Borell-Brascamp-Lieb inequality (see \cite{Borell,BL,G}), which concerns so-called {\em $\alpha$-concave} measures. A Borel measure $\mu$ defined on $\R^n$ is said to be $\alpha$-concave, for some $\alpha \in [-\infty,\infty]$, if, for all Borel measurable sets $A,B \subset \R^n$ and any $\lambda \in [0,1]$,
\begin{equation}\label{e:alphaconcave}
\mu((1-\lambda)A + \lambda B) \geq M_{\alpha}^{\lambda}(\mu(A), \mu(B)). 
\end{equation}
Analogously, a non-negative Borel measurable function $f$ defined on $\R^n$ is said to be $\alpha$-concave, for some $\alpha \in [-\infty, \infty],$ if
\[
f((1-\lambda)x + \lambda y) \geq M_{\alpha}^{\lambda}(f(x),f(y))
\]
for all $x,y \in \R^n$ and $\lambda \in [0,1]$. Here $M^{\lambda}_{\alpha}$ denotes the
{\em $\alpha$-mean} of two non-negative numbers:
\[
M_{\alpha}^{\lambda}(a,b)=\left\{
\begin{array}{ll}
\bigl((1-\lambda)a^{\alpha}+\lambda b^{\alpha}\bigr)^{\frac{1}{\alpha}}, & \text{ if }\alpha\neq 0,\pm\infty,\\[1mm]
a^{1-\lambda}b^{\lambda}, & \text{ if }\alpha=0,\\[1mm]
\max\{a,b\}, & \text{ if }\alpha=\infty,\\[1mm]
\min\{a,b\}, & \text{ if }\alpha=-\infty;
\end{array}\right.
\]
for $ab>0$;  $M_{\alpha}^{\lambda}(a,b)=0$ when $ab=0$. A $0$-concave function is usually called
\emph{log-concave} whereas a $(-\infty)$-concave function is called
\emph{quasi-concave}. Equivalently, a non-negative function $f$ defined on $\R^n$ is quasi-concave if each of its super-level sets 
$$C_t(f) = \{x \in \R^n \colon f(x) \geq t \|f\|_{\infty}\}$$
are convex sets for all $0 \leq t \leq 1 $. Here
$$
\|f\|_{\infty} = \inf \left\{ t \in \R \colon \vol_n(\{x \in \R^n \colon f(x) > t \}) = 0 \right\}
$$
denotes the {\em essential supremum} of $f$. One form of the  Borell-Brascamp-Lieb inequality is stated as follows (see \cite[Theorem~10.2]{G} or \cite[Proposition~1.4.4]{AGM}). 

\begin{theorem}[Borell-Brascamp-Lieb inequality] Let $\lambda \in [0,1]$ and $-\frac{1}{n} \leq \alpha \leq \infty$. Given non-negative measurable functions $f,g,$ and $h$ defined on $\R^n$ satisfying 
	\[
	h((1-\lambda)x) + \lambda y) \geq M_{\alpha}^{\lambda}(f(x),g(y))
	\]
	for all $x,y \in \R^n$, then 
	\begin{equation}\label{e:BBL}
	\int_{\R^n} h(x) dx \geq M_{\frac{\alpha}{n\alpha+1}}^{\lambda} \left(\int_{\R^n} f(x) dx, \int_{\R^n} g(x)dx \right).
	\end{equation}
\end{theorem}

Recently, inequality (\ref{e:RS}) was extended, after a suitable change accounting for the lack of translation invariance of general measures, to the setting of measures having radially decreasing densities (see \cite[Theorem~1.1]{ACRYZ}).  We say a function $\phi \colon \R^n \to \mathbb{R}_+$ is {\it radially decreasing} if, for each $t \in [0,1]$ and any $x \in \R^n$, one has 	$\phi(tx) \geq \phi(x)$.  Note that a quasi-concave function $f$ that assumes it maximum at the origin is radially decreasing. 

Following \cite{Co} we consider a functional analogue of the difference body. Given $\alpha$-concave function $f,g \colon \R^n \to \mathbb{R}_+$, for some $\alpha \in [-\infty,\infty]$, we define
\begin{equation*}
\Delta_{\alpha}^{f,g}(x) = \sup_{x=x_1-x_2} M_{\alpha}(f(x_1),g(x_2)),
\end{equation*}
where, for $a,b \geq 0$ with $ab>0$, 
\[
M_{\alpha}(a,b)=\left\{
\begin{array}{ll}
\bigl(a^{\alpha}+b^{\alpha}\bigr)^{\frac{1}{\alpha}}, & \text{ if }\alpha\neq 0,\pm\infty,\\[1mm]
ab, & \text{ if }\alpha=0,\\[1mm]
\max\{a,b\}, & \text{ if }\alpha=\infty,\\[1mm]
\min\{a,b\}, & \text{ if }\alpha=-\infty;
\end{array}\right.
\]
and  $M_{\alpha}(a,b)=0$ when $ab=0$.
The function $\Delta_{\alpha}^{f,g} $ is called the $\alpha$-{\it difference function}. This function is even and $\alpha$-concave for any $\alpha$. For more details on such functions, please see \cite{BC,Co,KL,Sch}. In particular, when $g=f$, we denote the $\alpha$-difference function by $\Delta_{\alpha} f$. In \cite{Co}, Colesanti established the following functional version of inequality (\ref{e:RS}),  in the case when $\alpha = \frac{1}{s}$ for some $s \in (-\infty,0]$:
\begin{equation}\label{e:Colesanti1}
\int_{\R^n} \Delta_{\frac{1}{s}} f(x) dx \leq \binom{2n}{n} \int_{\R^n} f(x) dx,
\end{equation}
where $f \colon \R^n \to \mathbb{R}_+$ is an integrable $\left(\frac{1}{s}\right)$-concave function.  Moreover, if one takes $f = \chi_K$, the characteristic function of $K$, then (\ref{e:RS}) is recovered.

Given a Borel measure $\mu$ on $\R^n$ with density $\phi$ and $H \in G_{n,m}$, we define the {\em marginal} of $\mu$ with respect to the subspace $H$ by 
\[
\mu(A \cap H) = \int_H \phi(x) \chi_A(x)dx
\]
for all compact subsets $A$ of $\R^n$. 

Let $\mu$ be a symmetric $\alpha$-concave measure on $\R^n$. If $K \subset \R^n$ is taken to be an origin-symmetric convex body and $H \in G_{n,m}$, then inequality \eqref{e:alphaconcave} together with the convexity of $K$ implies that the function
\[
x \in H^{\perp} \mapsto \mu(K \cap (x+H))^{\alpha}
\]
is an even concave function on its support. In particular, its maximum occurs at the origin; that is, 
\[
\max_{x \in H^{\perp}} \mu(K \cap(x+H))^{\alpha} = \mu(K \cap H)^{\alpha}
\]
holds whenever $\mu$ is a symmetric, $\alpha$-concave measure and $K \subset \R^n$ is an origin symmetric convex body. In the case when $\mu = \vol_m$, the previous inequality states that, for any origin-symmetric convex body in $\R^n$, the section of largest $m$-dimensional volume is the central section.

In \cite[Theorem~1]{RU1} Rudelson found an asymptotic inequality that measures the section of largest $m$-dimensional volume of a (not necessarily  symmetric) convex body of the same type as \eqref{e:RS}; it bounds the volume of the central section of the difference body of a convex body by an $m$-dimensional subspace from above by a constant multiple, depending both on the dimension and sub-dimension, of the maximal parallel section of the original body. This result reads as follows.

\begin{theorem}[Rudelson]\label{t:Rudelson_original} Given a convex body $K$ and $H \in G_{n,m}$, one has 
	\begin{equation}\label{e:RU}
	\vol_m((K-K) \cap H) \leq \left[c \psi(n,m) \right]^m \sup_{x \in \R^n} \vol_m(K \cap (H+x)), 
	\end{equation}
	where $c>1$ is some absolute constant and
	\[
	\psi(n,m) =\min \left\{ \frac{n}{m}, \sqrt{m} \right\}.
	\]
\end{theorem}

Inequality (\ref{e:RU}) was an important tool in estimating the Banach-Mazur distance between non-symmetric convex bodies and estimating the so-called $MM^*$-estimate for non-symmetric convex bodies (see \cite{AGM,RU2} for more details).

Applying inequality (\ref{e:RU}) to the identity (which follows form Fubini's theorem) 
\[
\frac{\int_{H} \Delta_{-\infty} f(x) dx}{\|f\|_{\infty} } = \int_0^1 \vol_m((C_t(f) - C_t(f)) \cap H)dt,
\]
and using the fact that, for all $a,b >0$, we have that $\left(a^{\frac{1}{s}}+b^{\frac{1}{s}}\right)^s \leq \min\{a,b\}$ whenever $s \in (-\infty, 0)$, $\Delta_{\frac{1}{s}}f \leq \Delta_{-\infty}f $ for all $s \in (-\infty,0)$, we extend inequality (\ref{e:Colesanti1}) to marginals of integrable quasi-concave functions. 

\begin{corollary}\label{t:RUQCFunctional} Given any integrable, bounded $\left(\frac{1}{s}\right)$-concave function $f \colon \R^n \to \mathbb{R}_+$ with $s \in (-\infty,0]$, 
	\begin{equation}\label{e:RU_F_Lebesgue}
	\frac{\int_{H} \Delta_{\frac{1}{s}} f(x) dx}{\|f\|_{\infty} }\leq  \left[c \psi(n,m) \right]^m \int_0^1 \sup_{y \in \R^n} \vol_m(C_t(f) \cap (H+y)) dt
	\end{equation}
	for some constant $c>0$.
\end{corollary}

One may wish to strengthen the inequality appearing in (\ref{e:RU_F_Lebesgue}), in the sense of commuting the integral with the supremum. We address this issue in the case of logarithmically concave functions (cf. Theorem~\ref{t:KMRlogconcave}).

Fix any $p \in \N$. Given a convex body $K$, we consider the  $np$-dimensional convex body given by 
\[
D_p(K)= \left\{\bar{x} = (x_1,\dots,x_p) \in (\R^n)^p \colon K \cap \left( \bigcap_{i=1}^p(x_i+K) \right) \neq \emptyset \right\}.
\]
Note that $D_1(K) = K-K$ is the usual difference body of $K$. 
These bodies were originally introduced by Schneider in \cite{Sch2}, where the convexity of the body $D_p(K)$ was established as well as the following Rogers-Shephard type inequality for $D_p(K)$: given a convex body $K \subset \R^n$, 
\begin{equation}\label{e:SRS}
\vol_{np}(D_p(K)) \leq \binom{np+n}{n} \vol_n(K)^p
\end{equation}
with equality if and only if $K$ is a simplex. 

The following theorem is the main result of this paper, which generalizes inequality \eqref{e:SRS}, and by extension Theorem~\ref{t:Rudelson_original} when $\psi(n,m) = \frac{n}{m}$, to the setting of measures with radially decreasing densities.

\begin{theorem}\label{t:SRRS} Fix $p \in \N$. Let $\eta$ be a measure on $\R^n$ given by $d\eta(x) = \psi(x)dx$, where $\psi \colon \R^n \to \R_+$ is $\left(\frac{1}{s}\right)$-concave, for some $s \in (0,\infty)$, and such that $\psi(0) = \|\psi\|_{\infty}$. For each $i=1,\dots,p$ let $\mu_i$ be measure on $\R^n$ with density $\phi_i \colon \R^n \to \R_+$ that is radially decreasing. Let $\nu = \prod_{i=1}^p \mu_i$ be the associated product measure on $(\R^n)^p$ having density $\phi$. For each $i=1,\dots,p$ let $H_i \in G_{n,m_i}$ $H_i  \in G_{n,m_i}$ be an $m_i$-dimensional subspace of the $i$th copy of $\R^n$, and set $\bar{H} = H_1 \times \cdots \times H_p$ be the associated product subspace of $(\R^n)^p$. Then, for any convex body $K \subset \R^n$ such that $\eta(K) >0$,
	\begin{equation}\label{e:SRRS}
	\begin{split}
	\nu\left(D_p(K) \cap \bar{H}\right) 
	\leq  \frac{c(n,m,s)}{\eta(K)} \int_{K} \prod_{i=1}^p \mu_i[(y-K) \cap H_i] d\eta(y),
	\end{split}
	\end{equation}
	where $m= m_1+\cdots+m_p$ and
	\[
	c(n,m,s) = \binom{n+m+s}{m+s}.
	\]
\end{theorem}

Here the combinatorial number $\binom{n+m+s}{m+s}$ for non-integer values of $s$ is defined in terms of the Beta function,  $B(x,y) = \int_0^1 (1-t)^{x-1}t^{y-1} dt$.. 
By choosing $p=1$, letting $m = 1,\dots, n$ be arbitrary,  and replacing $K$ with $-K$ in inequality \eqref{e:SRRS}, after an application of Stirling's formula we have the following extension of Theorem~\ref{t:Rudelson_original}.

\begin{corollary}
Let $\mu$ be any measure on $\R^n$ given by $\mu(x) = \phi(x) dx$, where $\phi \colon \R^n \to \R_+$ is radially decreasing and satisfies and let $H \in G_{n,m}$. Then, for any convex body $K \subset \R^n$ and any $s \in (0,\infty)$,
\begin{equation}\label{e:RudelsonGeneralized}
\begin{split}
\mu((K-K)\cap H) &\leq \binom{n+m+s}{m+s} \sup_{y \in \R^n} \mu(K \cap (H+y))\\
&\leq \left[C \cdot \frac{n+s}{m+s}\right]^{m+s} \sup_{y \in \R^n} \mu(K \cap (H+y)).
\end{split}
\end{equation}
\end{corollary}

\begin{remark}
Notice that, in the setting of Theorem \ref{t:SRRS}, as we will see in its proof, if we replace $d\eta(x)$ with the Lebesgue measure, inequality \eqref{e:SRRS} becomes 
	\begin{equation}\label{e:SRRSV}
	\nu\left(D_p(K) \cap \bar{H}\right) \leq  \frac{\binom{n+m}{m}}{\vol_n(K)} \int_{K} \prod_{i=1}^p \mu_i[(y-K) \cap H_i] dy,
	\end{equation}
\end{remark}

\begin{theorem}\label{t:RudelsonGeneralized}
	Let $\mu$ be a measure on $\R^n$ given by $d\mu(x) = \phi(x) dx$, where $\phi \colon \R^n \to \R_+$ is radially decreasing and let $H \in G_{n,m}$. Then, for any convex body $K \subset \R^n$,
	\begin{equation}\label{e:RU_good}
	\begin{split}
	\mu((K-K) \cap H) &\leq \binom{n+m}{m} \sup_{y \in \R^n}\mu(K \cap (H+y))\\
	&\leq \left[\frac{cn}{m}\right]^m \sup_{y \in \R^n} \mu( K \cap (H+y))
	\end{split}
	\end{equation}
	for some absolute constant $c>1$.
\end{theorem}

We would like to remark that, in the case when $\mu$ is taken to be a measure having an even quasi-concave density $\phi \colon \R^n \to \R_+$, then one can reverse ienquality \eqref{e:RU_good}. Indeed, noting that $\phi$ is a bounded, even, and quasi-concave functions, the super-level sets $C_t(\phi)$, are origin-symmetric convex bodies for all $0 \leq t < 1$.  For each $y \in \R^n$ set $K_t(y) = (K-y) \cap C_t(\phi)$. Consequently, using Fubini's theorem, together with the Brunn-Minkowki inequality, one may write, for each $y \in \R^n$, that
\begin{align*}
\mu(K \cap (H+y)) &= \mu((K-y) \cap H)\\
&= \|\phi\|_{\infty} \int_0^1 \vol_m(K_t(y) \cap H) dt\\
&\leq  \|\phi\|_{\infty} 2^{-m} \int_0^1 \vol_m((K_{t}(y)-K_t(y)) \cap H) dt\\
&\leq \mu\left(\left(\frac{K-K}{2}\right) \cap H \right),
\end{align*}
where, in the case inequality, we have used inclusion (which follows from the symmetry and convexity of $C_t(\phi)$ and the fact that $H$ is a subspace), 
\[
(K_t(y) - K_t(y)) \cap H \subset [(K-K) \cap 2C_t(\phi)] \cap H = 2 \left( \frac{K-K}{2} \cap C_t(\phi)\right) \cap H.
\]
This computation leads to the following result which measures the symmetry of sections of any convex body in terms of the central section of its associated difference body. 

\begin{theorem}
Suppose that $\mu$ is a measure on $\R^n$ with bounded, even quasi-concave density $\phi \colon \R^n \to \R_+$, and let $H \in G_{n.m}$. Then, for any convex body $K \subset \R^n$, one has 
\[
1 \leq \left[\frac{\mu((K-K) \cap H}{\sup_{y \in \R^n} \mu(K \cap (H+y))} \right]^{\frac{1}{m}}\leq \binom{n+m}{m}^{\frac{1}{m}} \leq C \cdot \frac{n}{m},
\]
where $C>1$ is some absolute constant.
\end{theorem}

For example, the above inequality implies that, for any convex body $K \subset \R^n$ and any $H \in G_{n,m}$ one has the following estimate for sections of $K$ with respect to the standard Guassian measure, $\gamma_n$, on $\R^n$:
\[
1 \leq \left[ \frac{\gamma_n((K-K) \cap H)}{\sup_{y\in \R^n} \gamma_n(K \cap (y+K)} \right]^{1/m} \leq c\frac{n}{m}
\]
for some absolute constant $c > 1$. Recall, that the standard Gaussian measure on $\R^n$ is the measure whose density is given by 
\[
g(x) = \frac{e^{-\frac{|x|^2}{2}}}{(2\pi)^{n/2}}.
\]

The proof of Theorem~\ref{t:SRRS} relies on the following theorem, which has its own independent interest, and implies additional inequalities of the Rogers-Shephard type, as we will see below.  This inequality is an extension of a theorem due to Chakerian (see \cite[Theorem~1]{Ch}).

\begin{theorem}\label{t:ChQuasiConcave} Let $\mu$ be a measure on $\R^n$ having a radially decreasing density, $H \in G_{n,m}$, and $h \colon \R^n \to \R_+$ be a strictly increasing differentiable function. Suppose that $g \colon \R^n \to \R_+$ is an integrable quasi-concave function, assuming its maximum at the origin, and suppose that $f \colon \R^n \to \R_+$ is a not identically zero concave function.  Then 
\begin{equation}\label{e:chf}
\int_H h(f(x))g(x) d\mu(x) \geq \beta \cdot \int_H g(x) d\mu(x).
\end{equation}
where 
\[
\beta = f(0) \int_0^1 h(tf(0))(1-t)^mdt.
\]
\end{theorem}

The organization of the paper is as follows. In the first part of Section 2, we prove both Theorem~\ref{t:SRRS} and Theorem~\ref{t:ChQuasiConcave}; in the second part of this section, we discuss addition sectional inequalities of the Rogers-Shephard type. As a consequence of the main theorem, in Section 3 marginal inequalities of the Rogers-Shephard type, first for $\left(\frac{1}{s}\right)$-concave functions, with $s \in \mathbb{Q}$, with $s \geq n$, and then, using a different method, we prove a variation of Corollary~\ref{t:RUQCFunctional} for logarithmically concave functions.

\section{Sectional Rogers-Shephard type inequalities}

This section is dedicated to the proof of Theorem~\ref{t:SRRS} and some additional inequalities of the Rogers-Shephard type. 

\subsection{Proof of the main theorem}

The proof of Theorem~\ref{t:SRRS} relies on the following version of Theorem~\ref{t:ChQuasiConcave}.

\begin{theorem}\label{t:chak}
Let $\mu$ be a measure on $\R^n$ having radially decreasing density $\phi \colon \R^n \to \R_+$. Let $f \colon K \to \R_+$ be a not identically zero concave function supported on a convex body $K \subset \R^n$ having the origin as an interior point. Let $h \colon \R_+ \to \R_+$ be a differentiable strictly increasing function.  Then, for any $H \in G_{n,m}$,
\begin{equation}\label{e:ch}
\int_{K \cap H} h(f(x)) d\mu(x) \geq \beta \cdot \mu(K \cap H),
\end{equation}
where 
\[
\beta = f(0) \int_0^1 h'(tf(0))(1-t)^m dt.
\]
\end{theorem}

\begin{proof} Integrating in polar coordinates, we may write 
\[
\int_{K \cap H} h(f(x)) d\mu(x) = \int_{\s^{n-1} \cap H}\int_0^{\rho_K(u)} h(f(ru)) \phi(ru) r^{m-1} dr du. 
\]
Consider the function $g\colon K \cap H \to \R_+$ given by \[
g(x) = f(0)\left(1 - \frac{|x|}{\rho_K\left(\frac{x}{|x|}\right)}\right) \quad \text{if } x \neq 0,
\]
and $g(0) = f(0)$.  The concavity of $f$, together with the fact that $f(0)=g(0)$, the monotonicity of the integral and fact that $h$ is strictly increasing implies that
\begin{equation}\label{e:ch1}
\begin{split}
\int_{K \cap H} h(f(x)) d\mu(x) & \geq \int_{\s^{n-1} \cap H} \int_0^{\rho_K(u)} h(g(ru))\phi(ru)r^{m-1} dr du.\\
\end{split}
\end{equation}

Fix some direction $u \in \s^{n-1} \cap H$ and consider the function $\psi \colon (0,\rho_K(u)] \to \R_+$ given by 
\[
\psi(y) = \beta \int_0^y \phi(r) r^{m-1} dr - \int_0^y h\left(f(0)\left(1 - \frac{r}{y}\right)\right)\phi(r) r^{m-1} dr,
\]
where $\beta > 0$ is a constant to be chosen such that $\psi(y) \leq 0$. Using the fact that $\phi$ is bounded together with the fact that $h$ is integrable on each segment $(0,y] \subset (0,\infty)$ (indeed, since $h$ is strictly increasing and differentiable, it is a continuous function that is bounded on each segment by an integrable function), we may assert that $\psi(y) \to 0$ as $y \to 0^+$. Since $\psi$ is absolutely continuous on each $[a,b] \subset (0,y]$, $\psi$ may be represented by 
\[
\psi(y) = \psi(a) + \int_0^y \psi'(s) ds.
\]
Consequently, to have $\beta >0$ to be such that $\psi(y) \leq 0$, it suffices for $\beta$ to be selected so that $\psi'(y) \leq 0$ for almost every $y \in (0,\rho_K(u)]$. Differentiation of $\psi$ yields the representation
\[
\psi'(y) = \beta \phi(y) y^{m-1} -f(0) \int_0^y h'\left(f(0)\left(1 - \frac{r}{y}\right)\right)\frac{r^m}{y^2} \phi(r) dr. 
\]
Since $\phi$ is radially decreasing, it suffices to select $\beta$ to satisfy the condition 
\[
\beta \leq f(0)\int_0^y h'\left(f(0)\left(1 - \frac{r}{y}\right)\right)\frac{r^m}{y^{m+1}} dr,
\]
or equivalently, applying the change of variables $u=r/y$ followed by the change of variables $t= 1-u$, we see that it suffices for $\beta$ to satisfy
\[
\beta \leq f(0) \int_0^1 h'(f(0)t) (1-t)^m dt.
\]
Choosing $\beta=f(0) \int_0^1 h'(f(0)t) (1-t)^m dt$, \eqref{e:ch1} implies that 
\begin{align*}
\int_{K \cap H} h(f(x)) d\mu(x) &\geq  \beta \int_{\s^{n-1}} \int_0^{\rho_K(x)} \phi(ru) r^{m-1} dr du\\
&= \beta \cdot \mu(K \cap H),
\end{align*}
as desired.
\end{proof}

We would like to remark that the full statement of Theorem~\ref{t:ChQuasiConcave}, follows from the above proof by noting that $C_s(g)$ (where $g$ is as in the statement of the theorem) are convex bodies containing the origin as an interior point for all $s \in (0,1)$, and an application of Fubini's theorem at the right moments. In this way inequality \eqref{e:ch} implies inequality \eqref{e:chf}.

We are now ready to proceed to the proof of Theorem \ref{t:SRRS}.

\begin{proof}[Proof of Theorem~\ref{t:SRRS}] Consider the function $f \colon (\R^n)^p \to \R_+$ given by 
	\begin{equation}\label{e:zero}
	f(\bar{x}) := f(x_1,\dots,x_p) = \eta\left[K \cap \left( \bigcap_{i=1}^p (x_i+K) \right) \right].
	\end{equation}
	We notice that $f$ is supported on $D_p(K)$ and vanishes on the boundary of $D_p(K)$. We claim that $f$ is $\left( \frac{1}{n+s} \right)$-concave on its support. Let $\bar{x}, \bar{y}\in D_p(K)$ and $\lambda \in [0,1]$ be arbitrary. 
	We need to show that 
	\begin{equation}\label{e:one}
	f((1-\lambda)\bar{x} + \lambda \bar{y}) \geq M_{\frac{1}{n+s}}^{\lambda} (f(\bar{x}),f(\bar{y})).
	\end{equation}
	In view of the Borell-Brascamp Lieb inequality \eqref{e:BBL}, in order to prove inequality \eqref{e:one}, it is sufficient only to verify that following inclusion holds:
	\begin{equation}\label{e:two}
	K^{\lambda}(x,y) \supset (1-\lambda)\left[K \cap \bigcap_{i=1}^p(x_i+K)\right] + \lambda\left[K \cap \bigcap_{i=1}^p(y_i+K)\right],
	\end{equation}
	where 
	\[
	K^{\lambda}(x,y) = K\cap \bigcap_{i=1}^p \left[(1-\lambda) x_i + \lambda y_i + K \right].
	\]
	Let $\bar{z}\in (1-\lambda)[K \cap \bigcap_{i=1}^p(x_i+K)] + \lambda[K \cap \bigcap_{i=1}^p(y_i+K)]$ be arbitrary. Then $\bar{z}= (1-\lambda)z+\lambda z'$ for some $z \in K \cap \bigcap_{i=1}^p(x_i+K)$ and $z' \in K \cap \bigcap_{i=1}^p(y_i+K)$. Using the convexity of $K$, we see that $\bar{z} \in K$. For each fixed $i=1,\dots, p$ there exist $k_i,k_i' \in K$ such that $
	z = x_i +k_i $ and $z' = y_i + k_i', $
	and so $\bar{z} = (1-\lambda) x_i + \lambda y_i + [(1-\lambda)k_i +\lambda k_i']$ for every $i$. Consequently, using the convexity of $K$ once again, it follows that $\bar{z} \in K^{\lambda}(x,y)$, and so the inclusion \eqref{e:two} follows. Hence, $f$ is $\left(\frac{1}{n+s}\right)$-concave on its support, as claimed.
	
	The main goal of the proof is to estimate the following integral from above and below:
	\[
	I(f) := \int_{\bar{H}} f(\bar{x}) d\nu(\bar{x}).
	\]
Notice that an application of Fubini's theorem allows us to write 
	\begin{equation}\label{e:three}
	\begin{split}
	I(f) &= \int_{\bar{H}} f(\bar{x}) d\nu(\bar{x})\\
	&=  \int_{H_1} \cdots \int_{H_p} \left( \int_{K} \prod_{i=1}^p \chi_{y-K}(x_i)d\eta(y)\right) d\mu_p(x_p) \cdots d\mu_1(x_1)\\
	&=\int_{K} \prod_{i=1}^p \mu_i((y-K) \cap H_i) d\eta(y)\\
	&= \eta(K)\frac{1}{\eta(K)}\int_{K} \prod_{i=1}^p \mu_i((y-K) \cap H_i) d\eta(y).
	\end{split}
	\end{equation}
	
	Finally, by applying inequality \eqref{e:ch} with the concave function $f^{\frac{1}{n+s}}$, the increasing function $s \mapsto s^{n+s}$, and the measures $\nu$ yields
\begin{equation}\label{e:yepyep}
\begin{split}
I(f) &\geq \eta(K) \nu(D_p(K) \cap \bar{H}) (n+s) \int_0^1 t^{(n+s)-1}t^m dt\\
&= (n+s) B(n+s,m+1) \eta(K) \nu(D_p(K) \cap \bar{H}).
\end{split}
\end{equation}
Combining \eqref{e:three} and \eqref{e:yepyep} completes the proof. 
\end{proof}

\begin{remark}[]
	In place of $D_p(K)$, we may instead consider the following: given convex bodies $K, L_1,\dots,L_p \subset \R^n$ with $\text{int}(K \cap (-L_1) \cap \dots \cap (-L_p))$ containing the origin, we define the convex (see \eqref{e:two} but with $-K$ replaced with $L_i$ in the intersection) $np$-dimensional set given by 
	\[
	D_p(K,L_i) := \left\{\bar{x}:=(x_1,\dots,x_p) \colon K \cap \left( \bigcap_{i=1}^p(x_i - L_i) \right)\neq \emptyset \right\}.
	\]
	In this setting (\ref{e:SRS}) becomes 
	\[
	\vol_{np}(D_p(K,L_i)) \leq \binom{np+n}{n} \frac{\vol_n(K)  \vol_n(L_1) \cdots \vol_n(L_p)}{\vol_n(K \cap \bigcap_{i=1}^p (-L_i))}.
	\]
	
	Let $ k \in \{1,\dots,n\}$. In the same setting at Theorem \ref{t:SRRS}, but replacing the function $f$ in \eqref{e:zero} with the function $\tilde{f}$ given by 
	\[
	\tilde{f}(x_1,\dots,x_p) = \vol_n \left[K \cap \left( \bigcap_{i=1}^p(x_i - L_i) \right) \right],
	\]
	we may repeat the proof to obtain the estimate
	\begin{equation}\label{e:oneze}
	\nu(D_p(K,L_i) \cap \bar{H}) \leq \frac{\binom{n+m}{n} \int_K \prod_{i=1}^p \mu_i(L_i \cap (y +H_i))dy}{\vol_n\left(K \cap \bigcap_{i=1}^p (-L_i)\right)}, 
	\end{equation}
	where $m= m_1+m_2+\cdots+m_p$ and $m_i \in \{1,\dots,n\}$ for all $i$. 
\end{remark}

As an immediate consequence of inequality \eqref{e:oneze}, we obtain the following: 

\begin{corollary} Let $\mu$ be a measure on $\R^n$ given by $d\mu(x) = \phi(x) dx$, where $\phi \colon \R^n \to \R_+$ is radially decreasing, $m  \in\{1,\dots,n\}$, and $H \in G_{n,m}$. Then, for any convex bodies $K, L\subset \R^n$ with $ 0 \in \text{int}(K \cap (-L))$, 
\begin{equation}\label{e:RSRKL}
	\mu((K+L) \cap H) \leq \frac{ \binom{n+m}{m} \int_K \mu(L \cap (H+y))dy}{\vol_n(K \cap (-L))}.
\end{equation}
\end{corollary}

\subsection{Additional inequalities of the Rogers-Shephard type}

In this subsection we collect additional consequences of Theorem~\ref{t:ChQuasiConcave} of the Rogers-Shephard type.

Looking more closely at the proof of Theorem~\ref{t:SRRS}, we notice that the two critical ingredients for finding the lower bound (the use of Theorem~\ref{t:ChQuasiConcave}) was a Brunn-Minkowski type inequality for the averaging measure with respect to a strictly increasing differentiable function for some class of convex bodies. With this in mind, analogously \cite[Theorem~3.8]{Li}, we have the following definition.

\begin{definition}
Let $Z \colon \R_+ \to \R$ be a strictly increasing differentiable function. We say that a Borel measure $\mu$ on $\R^n$ is {\bf $Z(t)$-concave} with respect to some class of Borel sets if, for every pair of sets $A,B \subset \R^n$ belonging to this class, and for every $\lambda \in [0,1]$, the following Brunn-Minkowski type inequality holds
\[
Z(\mu((1-\lambda)A + \lambda B)) \geq (1-\lambda)Z(\mu(A)) + \lambda Z(\mu(B)).
\]
\end{definition}

By repeating the proof of Theorem~\ref{t:SRRS}, but replacing the use of Borel-Brascamp-Lieb inequality \eqref{e:BBL} with the definition of a $Z(t)$-concave measure, we obtain the following theorem. 

\begin{theorem}\label{t:FinalRudelson}
Suppose that $Z(t) \colon \R_+ \to \R_+$ is strictly increasing and differentiable, and let $\eta$ be a $Z(t)$-concave measure on $\R^n$ with respect to a class of convex bodies in $\R^n$ that is closed under taking intersections and translations. Fix $p \in \N$. For each $i=1,\dots,p$ let $\mu_i$ be measure on $\R^n$ with density $\phi_i \colon \R^n \to \R_+$ that is radially decreasing. Let $\nu = \prod_{i=1}^p \mu_i$ be the associated product measure on $(\R^n)^p$ having density $\phi$. For each $i=1,\dots,p$ let $H_i \in G_{n,m_i}$ be an $m_i$-dimensional subspace of the $i$th copy of $\R^n$, and set $\bar{H} = H_1 \times \cdots \times H_p$ be the associated product subspace of $(\R^n)^p$. Then, for any convex body $K \subset \R^n$ belonging to this class such that $\eta(K) > 0$,
\begin{equation}\label{e:SSRSZ}
	\tilde{C} \cdot \nu\left(D_p(K) \cap \bar{H}\right) 
	\leq  \frac{1}{Z(\eta(K))} \int_{K} \prod_{i=1}^p \mu_i[(y-K) \cap H_i] d\eta(y),
\end{equation}
	where $m= m_1+\cdots+m_p,$ $K = K - g_{K,\eta}$, and 
	\[
	\tilde{C} = \int_0^1 (Z^{-1})'[tZ(\eta(K))](1-t)^mdt.
	\]
	Here $(Z^{-1})'$ stands for the derivative of inverse of the function $Z$. 
\end{theorem}

We have the following immediate corollary:

\begin{corollary}
Suppose that $Z(t) \colon \R_+ \to \R_+$ is strictly increasing and differentiable, and let $\eta$ be a $Z(t)$-concave measure on $\R^n$ with respect to a class of convex bodies in $\R^n$ that is closed under taking intersections and translations. Let $\mu$ be a measure on $\R^n$ having a radially decreasing density $\phi \colon \R^n \to \R_+$, and let $H \in G_{n,m}$. Then, for all convex bodies $K \subset \R^n$ belonging to this class with $\eta(K) > 0$, we have 
\[
 \tilde{C} \cdot \mu((K-K) \cap H) \leq \frac{1}{Z(\eta(K))} \int_{K} \mu(K \cap (H+y)) d\eta(y),
\]
where 
\[
\tilde{C} = \int_0^1 (Z^{-1})'[tZ(\eta(K))](1-t)^mdt.
\]
\end{corollary}

\begin{remark}
We note that the class of $Z(t)$-concave Borel measures is the largest class of measures for which Rogers-Shephard type inequalities of the form \eqref{e:RudelsonGeneralized} hold. Indeed, in inequality \eqref{e:RudelsonGeneralized}, we notice that as $s \to -m$ the constants $\binom{n+m+s}{m+s} \to 1$. However, it was shown by Rudelson  (see \cite[Theorem~2]{RU1}), that the quantity
\[
\left[ \frac{\vol_m((K-K) \cap H)}{\sup_{y \in \R^n} \vol_m(K \cap (x+H))}\right]^{1/m}
\]
is not uniformly bounded in general.
\end{remark}

To see the usefulness of the above theorem, we consider the following example of such a function $Z(t)$ and a measure $\mu$ that are $Z(t)$-concave with respect to a class of convex bodies. 

\begin{example}
Following \cite{RYN} we will say a non-empty measurable subset $A$ of $\R^n$ is {\it weakly unconditional} if, for every point $x = (x_1,\dots, x_n)$ belonging to $A$, the point $\epsilon x =(\epsilon_1 x_1, \dots, \epsilon_n x_n)$ belongs to $A$ for every $\epsilon:=(\epsilon_1, \dots, \epsilon_n) \in \{0,1\}^n$. In this paper, the authors established the inequality 
\begin{equation}\label{e:dimenBM}
\mu((1-\lambda) K + \lambda L)^{\frac{1}{n}} \geq (1-\lambda) \mu(K)^{\frac{1}{n}} + \lambda \mu(L)^{\frac{1}{n}}. 
\end{equation}
for any measure $\mu = \mu_1 \times \cdots \times \mu_n$ on $\R^n$, where each $\mu_i$ is a measure on $\R^n$ having a radially decreasing density, and $K,L$ are weakly unconditional sets such that $(1-\lambda)A + \lambda B$ is measurable.  In this setting, Theorem~\ref{t:FinalRudelson}, when $p=1$, can be acquire with the class of weakly unconditional convex bodies, the measure $\mu$, and the strictly increasing differentiable function $Z(t) = t^{1/n}$. 
\end{example}

We conclude this section with another application of Theorem~\ref{t:ChQuasiConcave}, 

\begin{theorem}\label{t:PSZconcave}
Let $m \in\{1,\dots,n-1\}$. Let $\eta$ be a measure on $\R^m$ that is $Z(t)$-concave on a class on convex bodies containing the origin for some strictly increasing differentiable function $Z \colon \R_+ \to \R_+$, and let $\mu$ be a measure on $\R^{n-m}$ whose density $\phi \colon \R^{n-m} \to \R_+$ is radially decreasing. Let $\nu$ be the product measure $\nu = \mu \times \eta$ on $\R^n$.  Then, for any convex body $K$ belonging to this class and for any $H \in G_{n,m}$,
\begin{equation}\label{e:PSZ}
\bar{C} \cdot \mu(P_{H^{\perp}} K) \leq \nu(K), 
\end{equation}
where
\[
\bar{C} = Z(\eta(K  \cap H)) \int_0^1 (Z^{-1})'[tZ(\eta(K))](1-t)^m dt.
\]
Here $P_{H^{\perp}}K$ is the orthogonal projection of $K$ onto the orthogonal complement $H^{\perp}$ of $H$, and $[Z^{-1}]'$ denotes derivative of the inverse of the function $Z$. 
\end{theorem}

\begin{proof}
Using to convexity of $K$, together with the assumption that the measure $\eta$ is $Z(t)$-concave, the function $f \colon H^{\perp} \to \R_+$ given by 
\[
f(x) = \eta(K \cap (H+x))
\] is $Z(t)$-concave on its support. Finally, applying inequality \eqref{e:ch1} with the concave function $Z(f)$, the increasing function $h(t) =Z^{-1}(t)$, and the measure $\mu$, we observe that 
\begin{align*}
\nu(K) &= \int_{P_{H^{\perp}}K} f(x) d\mu(x)\\
&\geq \bar{C} \cdot \mu(P_{H^{\perp}}K),
\end{align*}
as desired.
\end{proof}

By taking $\eta$ to be a measure that is $\frac{1}{m+s}$ concave and $Z(t) = t^{\frac{1}{m+s}}$, Theorem~\ref{t:PSZconcave} yields the following corollary (which originally appeared as Theorem~5.1 in \cite{ACRYZ}). 

\begin{corollary} Let $1 \leq m \leq n-1$. Let $\eta$ be a measure on $\R^{m}$ whose density $\psi \colon \R^m \to \R_+$ is $\left(\frac{1}{s}\right)$-concave for some $s > 0$ and let $\mu$ be a measure on $\R^{n-m}$ whose density $\phi \colon \R^{n-m} \to \R_+$ is radially decreasing. Let $\nu$ be the product measure $\nu = \mu \times \eta$ on $\R^n$. Let $K \subset \R^n$ be a convex body with the origin as an interior point. Then, for any $H \in G_{n,m}$,
\begin{equation*}
\mu(P_{H^{\perp}}K) \eta(K \cap H) \leq \binom{n+s}{n-m}\nu(K).
\end{equation*}
\end{corollary}

\section{Application: the functional setting}

In this section, we prove inequalities of the Rogers-Shephard type for functions that are either $\left(\frac{1}{s}\right)$-concave, with $0 \leq s < \infty$, or logarithmically concave; the first subsection concerns the former and the second subsection the latter.

\subsection{The case of $\left(\frac{1}{s} \right)$-concave functions, with $0 \leq s < \infty$}

\begin{theorem}\label{t:RU_s_concave_functional} Let $\mu$ be a Borel measure on $\R^n$ with radially decreasing density $\phi \colon \R^n \to \mathbb{R}_+$ and let $H \in G_{n,m}$.  Then, for any integrable, $\left(\frac{1}{s}\right)$-concave functions $f,g \colon \R^n \to \mathbb{R}_+$ with $s\in \N$, each of whose supports contain the origin as an interior point, one has  
	\begin{equation}\label{e:RudelsonsFunction}
	\begin{split}
&\left(\int_{\R^n} \min\{f(x),g(x)\} dx \right) \left( \int_{H} \Delta_{\frac{1}{s}}^{f,g}(x) d\mu(x)\right) \leq \\
&\leq C(n,m,s)\cdot \left( \int_{\R^n} f(x) dx \right) \left(\sup_{y \in \R^n} \int_{H-y} g(-x) d\mu(x)\right),
	\end{split}
	\end{equation}
where
\[
C(n,m,s) = \left[\frac{C(n+s)}{m+s} \right]^{m+s}
\]
for some absolute constant $C>0$. In particular, 
\begin{equation}\label{e:RudelsonsFunctions1}
\int_{H} \Delta_{\frac{1}{s}}f(x) d\mu(x) \leq C(n,m,s)\cdot \sup_{y \in \R^n} \int_{H+y} f(x) d\mu(x).
\end{equation}
\end{theorem}

\begin{proof} 
The proof of Theorem~\ref{t:RU_s_concave_functional} follows the ideas introduced by Klartag in \cite{KL} used in his proof of the Borell-Brascamp-Lieb inequality for $\left(\frac{1}{s}\right)$-concave functions with $s \geq 0$. 

Let $s$ be a positive integer. Given any bounded, integrable function $h \colon \R^n \to \mathbb{R}_+$, consider the set 
$$
A_{h,s} := \left\{(x,y) \in \R^n \times \R^s \colon x \in \text{supp}(h), |y| \leq h(x)^{\frac{1}{s}}\right\}.
$$
 Let $\nu$ be a Borel measure on $\R^n$ having density $\psi$. Consider the measure $\nu$ defined on $\R^n \times \R^s$ given by $d\nu_s(x,y) = \psi(x) dxdy.$ Then, for any $H \in G_{n,m}$, integrating in polar coordinates gives 
\begin{equation}\label{e:neat}
\begin{split}
\nu_s(A_{h,s} \cap (H \times \R^s)) &= \int_{\text{supp}(h) \cap H} \int_{|y| \leq h(x)^{\frac{1}{s}}}dyd\nu(x)\\
&= \int_{\text{supp}(h) \cap H} \int_{\s^{s-1}} \int_0^{h(x)^{\frac{1}{s}}}r^{s-1}drdud\nu(x)\\
&= \omega_s \int_{\text{supp}(h) \cap H}  h(x) d\nu(x),
\end{split}
\end{equation}
where $\omega_s = \vol_s(B_s)$. Moreover, we remark that $A_{h,s}$ is a convex body if and only if $h$ is an $\left(\frac{1}{s}\right)$-concave function.
	
 We begin by noticing that 
	\[A_{\Delta_{\frac{1}{s}}^{f,g}, s} =(A_{f,s} + (-A_{g,s})).
	\]
	Since $f,g$ are integrable $\left(\frac{1}{s}\right)$-concave functions, their respective support are convex bodies containing the origin in the interior of their intersection, and so $0 \in \text{int}(A_{f,s} \cap A_{g,s})$. Hence, applying inequality \eqref{e:RSRKL} to the pair of convex bodies $A_{f,s}$ and $A_{g,s}$, we obtain 
\begin{equation*}
\begin{split}
&\vol_{n+s}(A_{f,s} \cap A_{g,s}) \cdot \nu_s((A_{f,s} -A_{g,s})) \cap(H \times \R^s)) \leq \\
&\leq C(n,m,s) \cdot \vol_{n+s}(A_{f,s}) \cdot \sup_{(x,y) \in \R^n \times \R^s}\nu_s(-A_{g,s} \cap ((H \times \R^s) - (x,y))\\
&\leq C(n,m,s) \cdot \vol_{n+s}(A_{f,s}) \cdot \sup_{z \in \R^n} \nu_s(-A_{g,s} \cap ((H-z) \times \R^s).
\end{split}
\end{equation*}
Consequently, we obtain
\begin{equation}\label{e:anot}
\begin{split}
&\vol_{n+s}(A_{f,s} \cap A_{g,s}) \cdot \nu_s((A_{f,s} -A_{g,s})) \cap(H \times \R^s))\leq \\
&\leq C(n,m,s) \cdot \vol_{n+s}(A_{f,s}) \cdot \sup_{z \in \R^n} \nu_s(-A_{g,s} \cap ((H-z) \times \R^s).
\end{split}
\end{equation}

Observe, that in view of equality \eqref{e:neat}, we may express each of the quantities in inequality \eqref{e:anot}, from left-most to right-most, in the following way.
\begin{equation}\label{e:anot2}
\begin{split}
&\vol_{n+s}(A_{f,s} \cap A_{g,s}) = \omega_s \cdot \int_{\R^n}\min\{f(x),g(x)\} dx,\\
&\nu_s((A_{s,f} - A_{g,s}) \cap(H \times \R^s)) = \omega_s \cdot \int_H \Delta_{\frac{1}{s}}^{f,g}(x) d\mu(x),\\
&\vol_{n+s} (A_{f,s}) = \omega_s \cdot \int_{\R^n} f(x) dx,\\
&\sup_{z \in \R^n} \nu_s(-A_{g,s} \cap ((H-z) \times \R^s) = \omega_s \cdot \sup_{z \in \R^n} \int_{H-z} g(-x) d\mu(x).
\end{split}
\end{equation}
Combining \eqref{e:anot} and \eqref{e:anot2} establishes inequality \eqref{e:RudelsonsFunction}. 

Inequality \eqref{e:RudelsonsFunctions1} follows by taking $g(x) = f(-x)$ in inequality \eqref{e:RudelsonsFunction}.
\end{proof}

We notice that Theorem \ref{t:Rudelson_original} implies for any convex body $K \subset \R^n$ and any $H \in G_{n,m}$, where $m \in \{1,\dots,n\}$
\begin{equation}\label{e:yep}
\vol_m((K-K) \cap H) \leq \left(cn^{\frac{1}{3}}\right)^m \sup_{y \in \R^n} \vol_m(K \cap (H+y))
\end{equation}
for some absolute constant $c>1$. 
Repeating the proof of Theorem \ref{t:RU_s_concave_functional} in the case of the Lebesgue measure, but applying inequality \eqref{e:yep} to $A_{f,s}$ in place of \eqref{e:RSRKL}, we get the following theorem. 

\begin{theorem}	Let $f \colon \R^{n} \to \R_+$ be an  integrable $\left(\frac{1}{s}\right)$-concave function, for some $s \in \N$ and let $H \in G_{n,m}$, $m \in \{1,\dots,n\}$.  Then 
	\[
	\int_H \Delta_{\frac{1}{s}}f(x) dx \leq \left(c(n+s)^{\frac{1}{3}}\right)^{m+s} \sup_{y \in \R^n} \int_{H+y} f(x) dx,
	\]
	where $c>1$ is some absolute constant. 
\end{theorem}

To finish this section, we prove the following theorem, which is a marginal inequality of the Rogers-Shephard type for $(1/s)$-concave functions for general $0 \leq s < \infty$

\begin{theorem}\label{t:RudelsonFuntionalRationals}
Let $\mu$ be a measure on $\R^n$, with a radially decreasing density, let $0 \leq s < \infty$, and let $H \in G_{n,m}$.  Consider any  $\mu$-integrable, $\left(\frac{1}{s}\right)$-concave function $f \colon \R^n \to \R_+$, assuming its maximum at the origin and such that  $z_0 \in  \R^n$ satisfies
\begin{equation}\label{e:condition}
\sup_{z \in \R^n} \mu(\text{supp}(f) - z) \cap H) = \mu((\text{supp}(f) - z_0) \cap H)
\end{equation}
and $\text{supp}(f)-z_0$ contains the origin as an interior point. Then one has 
\[
 \int_H \Delta_{\frac{1}{s}} f(x) d\mu(x) \leq \left[\tilde{C} \cdot \frac{(n+1)(n+s)}{m(m+1)} \right]^{(m+s)} \sup_{y \in \R^n} \int_{H+y} f(x) d\mu(x),
\]
where $\tilde{C} >1$ is some absolute constant.
\end{theorem}

We would like to remark that the condition \eqref{e:condition} is not necessary in the case, when $\mu$ is taken to be the standard Gaussian measure on $\R^n$. 

\begin{proof} Assume that $s = p/q \in \mathbb{Q}$ is in lowest terms. We begin by noting that, since $f$ has a positive degree of concavity, the set $S_f = \{f >0\}$ is a convex body. Our goal will be to bound the following integral:
\[
\int_{(S_f -S_f)\cap H} \left[\Delta_{\frac{1}{s}}f(x)\right]^qd\mu(x).
\]
An application of the Jensen inequality yields
\begin{equation*}
\int_{(S_f - S_f) \cap H} \left[\Delta_{\frac{1}{s}}f(x)\right]^qd\mu(x) \geq \mu((S_f - S_f) \cap H)^{1-q} \left[\int_{(S_f-S_f) \cap H} \Delta_{\frac{1}{s}}f(x)d\mu(x) \right]^q,
\end{equation*}
or equivalently,
\begin{equation}\label{e:one1}
\begin{split}
\left[ \int_{H} \Delta_{\frac{1}{s}}f(x) d\mu(x) \right]^q \leq \mu((S_f-S_f) \cap H)^{q-1} \int_H \left[\Delta_{\frac{1}{s}}f(x)\right]^qd\mu(x)\\
\end{split}
\end{equation}

 We may apply inequality \eqref{e:RudelsonGeneralized} with $s=p/q$ and applying Stirling's formula, we see that 
\begin{equation}\label{e:two2}
\mu((S_f-S_f) \cap H) \leq \left[C_1 \cdot \frac{n+s}{m+s} \right]^{m+s} \sup_{z \in \R^n} \mu((S_f - z) \cap H).
\end{equation}
Since $[\Delta_{\frac{1}{s}}f]^q = \Delta_{\frac{1}{p}}f^q$, and because $f^q$ is $(1/p)$-concave,, we may apply inequality \eqref{e:RudelsonsFunctions1} to the function $f^q$ with the integer $p$ to obtain 
\begin{equation}\label{e:three3}
\begin{split}
\int_H \left[\Delta_{\frac{1}{s}}f(x)\right]^qd\mu(x) &\leq \left[C_2 \cdot \frac{n+p}{m+p} \right]^{m+p} \sup_{y \in \R^n} \int_{H+y} f(x)^q d\mu(x)\\
&\leq \left[C_2 \cdot \frac{n+1}{m+1} \right]^{q(m+s)} \sup_{y \in \R^n} \int_{H+y} f(x)^q d\mu(x).
\end{split}
\end{equation}
Inequalities \eqref{e:one1}, \eqref{e:two2}, and \eqref{e:three3} imply that 
\begin{equation}\label{e:four4}
\begin{split}
\left[\int_H \Delta_{\frac{1}{s}} f(x) d\mu(x)\right]^q
&\leq D(n,m,p,q) \left[ \sup_{z\in\R^n} \mu((S_f-z) \cap H)\right]^{q-1} \cdot \sup_{y \in \R^n}  \int_{H+y} f(x)^q d\mu(x) \\
&=  D(n,m,p,q) \mu((S_f-z_0) \cap H)^{q-1} \sup_{y \in \R^n} \int_{H+y} f(x)^q d\mu(x),
\end{split}
\end{equation}
where 
\[
D(n,m,p,q) = \left[C_1 \cdot \frac{n+s}{m+s} \right]^{q(m+s)} \cdot \left[C_2 \cdot \frac{n+1}{m+1} \right]^{q(m+s)}.
\]
From the assumption, it must be the case that the origin belongs to the interior of $K-y_0$, and so we may apply inequality \eqref{e:ch}  to the concave function $f^{q/p}$, the strictly increasing function $r \in \R_{+} \mapsto r^s$, the convex body $S_f-z_0$, and applying Stirling's formula, we see that 
\[
 \mu((S_f-y_0) \cap H) \leq \frac{\binom{m+s}{m}}{\|f\|_{\infty}} \int_{H+y} f(x) d\mu(x) \leq \frac{\left[C_3 \cdot \frac{m+s}{m} \right]^m}{\|f\|_{\infty}} \int_{H+y} f(x) d\mu(x).
\]
This, together with inequality \eqref{e:four4} yields 
\[
\left[\int_H \Delta_{\frac{1}{s}} f(x) d\mu(x)\right]^q
\leq E(n,m,p,q) \left[\sup_{w \in \R^n} \int_{H+w} f(x) d\mu(x) \right]^q,
\]
where
\begin{align*}
E(n,m,p,q) &= \left[C_1 \cdot \frac{n+s}{m+s} \right]^{q(m+s)}\left[C_2 \cdot \frac{n+1}{m+1} \right]^{q(m+s)}\left[ C_3 \cdot \frac{m+s}{m}\right]^{q(m+s)}\\
&= \left[C_4 \cdot \frac{(n+1)(n+s)}{m(m+1)} \right]^{q(m+s)}.
\end{align*}
Taking the $q$th root, we obtain the estimate 
\[
\int_H \Delta_{\frac{1}{s}} f(x) d\mu(x) \leq \tilde{E}(n,m,p,q) \sup_{w \in \R^n} \int_{H+w} f(x) d\mu(x),
\]
with $\tilde{E}(n,m,p,q) = E(n,m,p,q)^{1/q}$; it follows that
\[
\int_H \Delta_{\frac{1}{s}}f(x) d\mu(x) \leq \left[C_4 \cdot \frac{(n+1)(n+s)}{m(m+1)} \right]^{(m+s)} \sup_{w \in \R^n} \int_{H+w} f(x) d\mu(x),
\]
as desired. The case for general values of $0 \leq s < \infty$ follows by a standard approximation argument. 
\end{proof}

\subsection{The case of logarithmically concave functions}
We begin this section by defining the class of admissible functions,
$$
\text{LC}_0 = \left\{f \colon \R^n \to \mathbb{R}_+ \colon f \text{ is } \log\text{-concave, } f(0) = \|f\|_{\infty}, 0 < \int f<\infty \right\}.
$$
The main theorem of this section reads as follows. 

\begin{theorem}\label{t:KMRlogconcave}
	Let $f \in \text{LC}_0 $ and let $H \in G_{n,m}$. Then
	\begin{equation}\label{e:LC1}
	\left[\frac{\int_{H} \Delta_0 f(x) dx}{\sup_{y \in \R^n} \left\{\frac{\|f\|_{\infty}}{f_{H+y}}\int_{H+y}f(x)dx\right\}}\right]^{\frac{1}{m}}\leq C \|f\|_{\infty}^{1/m} \psi(n,m)
	\end{equation}
	where
	\[f_{H+y} :=\sup_{x \in H+y} f(x) \quad \text{and} \quad \psi(n,m) = \min\left\{\frac{n}{m}, \sqrt{m} \right\}.\]
Moreover,
\[
c\|f\|_{\infty} \leq \int_{H} \Delta_0 f(x) dx.
\]
Here $c> 0$ and $C>1$ are some absolute constants.
\end{theorem}

As an immediate consequence of inequality (\ref{e:LC1}), we have the following corollary.

\begin{corollary}
	Let $f \in \text{LC}_0 $ and let $H \in G_{n,m}$. Then
	\[
	\left[\frac{\int_{H} \Delta_0 f(x) dx}{\sup_{y \in \R^n} \left\{\frac{\|f\|_{\infty}}{f_{H+y}}\int_{H+y}f(x)dx\right\}}\right]^{\frac{1}{m}}\leq C \|f\|_{\infty}^{1/m} n^{1/3}.
	\]
	for some absolute constant $C>1$. 
\end{corollary}

Before proceeding to the proof of Theorem~\ref{t:KMRlogconcave}, we must first introduce some concepts that are critical to the proof. 

To functions $f \in \text{LC}_0 $, for each $m \in \{1,\dots,n\}$, one may associate the following $n$-dimensional convex body originally due to Ball (see \cite{Ba} and \cite{Bob}):
$$
K_m(f) = \left\{x \in \R^n \colon \left(\frac{1}{\|f\|_{\infty}} \int_0^{\infty} mr^{m-1}f(rx) dr \right)^{-\frac{1}{m}} \leq 1 \right\}.
$$
The radial function of $K_m(f)$ is given by 
$$
\rho_{K_m(f)}(u) = \left(\frac{1}{\|f\|_{\infty}} \int_0^{\infty} mr^{m-1}f(ru) dr \right)^{\frac{1}{m}},
$$
where $u \in \s^{n-1}$. Moreover, for any $H \in G_{n,m}$, one has 
$$
\int_H f(x) dx = \|f\|_{\infty} \vol_m(K_m(f)\cap H). 
$$
Indeed, integrating in polar coordinates, we see that
\begin{align*}
\vol_m(K_m(f)\cap H) &= \int_{H} \chi_{K_m(f)}(x) dx \\
&= \int_{\s^{n-1} \cap H } \int_0^{\rho_{K_m(f)}(u)}r^{m-1} dr du\\
&= \frac{1}{m}\int_{\s^{n-1} \cap H } \rho_{K_m(f)}(u)^m du\\
&=\frac{1}{m}\int_{\s^{n-1} \cap H } \frac{m}{\|f\|_{\infty}} \int_0^{\infty} f(ru) r^{m-1} dr du\\
&= \frac{1}{\|f\|_{\infty}} \int_{H} f(z) dz.\\
\end{align*}

Additionally, we will use the following axuillary lemmas due to Klartag and Milman (see \cite[Lemma~2.2]{KM} and \cite[Lemma 2.7]{KM}, respectively).  We include their proofs in the Appendix for completeness. For more information on such sets see also \cite{ABG}

For $f \in \text{LC}_0 $ and $1 \leq m \leq n$, define the set
$$
L_m(f) =\{x \in \R^n \colon f(x) \geq \|f\|_{\infty} e^{-m} \}. 
$$

\begin{lemma}\label{t:L2}
	Given $f \in \mathcal{L}^n$ and $1 \leq m \leq n$, 
	$$
	K_m(f) \subset L_m(f) \subset c K_m(f)
	$$
	for some absolute constant $c>1$. 
\end{lemma}

\begin{lemma}\label{t:L3}
	Let $f \in \text{LC}_0$ and let $1 \leq m \leq n$. Then 
	\[
	K_m(\Delta_0f) \subset c [K_m(f) + (-K_m(f))] \subset c' K_m(\Delta_0 f)
	\]
	for some absolute constants $c,c' > 1$. 
\end{lemma}

The power of the above lemmas is that they allow the replacement of the bodies $K_m(f)$ with certain level sets of a logarithmically concave function $f$. We now proceed to the proof of Theorem~\ref{t:KMRlogconcave}.

\begin{proof}[Proof of Theorem~\ref{t:KMRlogconcave}] Without loss of generality, we may assume that $\|f\|_{\infty} = 1$. For brevity we set $g = \Delta_0 f$ and consider its $m$-dimensional associated body $K_m(g)$. Since $\|\Delta_0f\|_{\infty} = \|f\|_{\infty}^2,$ we may also assume that $\|g\|_{\infty} = 1$. As mentioned above, we have that 
	\[\int_H g(x)dx = \vol_m(K_m(g) \cap H).
	\] Using Lemma~\ref{t:L3}, we must have that $K_m(g) \subset c(K_m(f)+ (-K_m(f)))$ for some constant $c>0$. Then, by applying (\ref{e:RU}), it follows that 
	\begin{equation}\label{e:lc1}
	\begin{split}
	\int_H g(x)dx &\leq  c^m \vol_m((K_m(f) +(-K_m(f))) \cap H)\\
	&\leq [c\psi(n,m)]^m \sup_{y \in \R^n} \vol_m(K_m(f) \cap (H+y)).
	\end{split}
	\end{equation}
	Fix an arbitrary $y \in \R^n$. We must compare $\vol_m(K_m(f) \cap (H+y))$ and $\int_{H+y} f(x)dx$. In view of Lemma~\ref{t:L2}, we observe 
	\begin{equation}\label{e:lc2}
	\begin{split}
	K_m(f) \cap (H+y) &\subset L_m(f) \cap (H+y)\\
	&= \{x \in H+y \colon f(x) \geq e^{-m}\}\\
	&\subset \{x \in H +y \colon f(x) \geq f_{H+y} e^{-m}\}\\
	&= L_m \left(f \mid_{H+y} \right)\\
	&\subset c' K_m \left( f \mid_{H+y} \right)
	\end{split}
	\end{equation}
	for some absolute constant $c'>0$. Let $\rho :=\rho_{K_m \left( f\mid_{H+y} \right)}$. Integrating in polar coordinates, 
	\begin{equation}\label{e:lc3}
	\begin{split}
	\vol_m\left(K_m \left( f\mid_{H+y} \right) \right) &= \int_{\R^n} \chi_{K_m \left( f\mid_{H+y} \right) }(x) dx \\
	&=  \int_{\partial B_y \cap (H+y) } \int_0^{\rho(u)}r^{m-1} dr du\\
	&=\frac{1}{m}\int_{\partial B_y \cap (H+y)} \rho(u)^m du\\
	&=\int_{\partial B_y \cap (H+y)} \frac{1}{f_{H+y}} \int_0^{\infty} f(rz) r^{m-1} dr du\\
	&= \frac{1}{f_{H+y}} \int_{H+y} f(x) dx,\\
	\end{split}
	\end{equation}
	where $B_y= B_n+y$ and $f_{H+y} = \sup_{x \in H+y} f(x)$. 
	Combining (\ref{e:lc1}), (\ref{e:lc2}), and (\ref{e:lc3}), we obtain
	\[
	\int_{H} g(x) dx \leq [c\psi(n,m)]^m \sup_{y \in \R^n} \left\{ \frac{1}{f_{H+y}} \int_{H+y} f(x) dx \right\}.
	\]
	Finally, by taking the $m$th root of both sides, we have proven the upper bound of (\ref{e:LC1}).  
	
	Now we prove the second inequality of the theorem. In view of Lemma~\ref{t:L2}, we may apply inclusions similar to \eqref{e:lc2} to conclude that 
	$$
	K_m\left(f|_{H}\right) \subset c_1 K_m(f) \cap H
	$$ 
	for some absolute constant $c_1>0$.  Using Lemma~\ref{t:L3} together with the Brunn-Minkowski inequality, we see that
	\begin{equation}\label{e:sure2}
	\begin{split}
	\vol_m\left(K_m\left(f |_{H}\right)\right) &\leq c_1^m \vol_m((K_m(f)\cap H)\\
	&\leq c_1^m 2^{-m} \vol_m((K_m(f) + (-K_m(f)) \cap H)\\
	&\leq (c')^m 2^{-m} \vol_m(K_m(g)\cap H)\\
	&= (c')^m 2^{-m} \int_H g(x) dx.
	\end{split}
	\end{equation}
	Combining \eqref{e:lc3} and \eqref{e:sure2}, we see that, for some constant $\tilde{c}>0$, 
	\begin{align*}
	(\tilde{c})^m\frac{1}{\|f\|_{\infty}} \int_{H} f(x) dx &\leq \vol_m(K_m(g)\cap H)\\
	&= \int_H g(x) dx.
	\end{align*}
Taking the $m$th root yields the lower bound of inequality (\ref{e:LC1}), completing the proof. 
\end{proof}

\section{Appendix}

Here we proof the auxillary results from \cite{KM} that were used to establish Theorem~\ref{t:KMRlogconcave}. The proofs seen below follow the ideas of \cite{KM}.
We again consider the following class of admissible functions,
$$
\mathcal{L}^n = \left\{f \colon \R^n \to \R_+ \colon f \text{ is } \log\text{-concave, } f(0) = ||f||_{\infty}, 0 < \int f<\infty \right\}.
$$
To functions $f \in \mathcal{L}^n$, for each $m \in \{1,\dots,n\}$, one my associate the following $m$-dimensional convex body
$$
K_m(f) = \left\{x \in \R^n \colon \left(\frac{1}{||f||_{\infty}} \int_0^{\infty} mr^{m-1}f(rx) dr \right)^{-1/m} \leq 1 \right\}
$$
whose radial function is given by 
$$
\rho_{K_m(f)}(u) = \left(\frac{1}{||f||_{\infty}} \int_0^{\infty} mr^{m-1}f(ru) dr, \right)^{1/m}
$$
where $u \in \s^{n-1}$. Moreover, one see that, for any $m$-dimensional linear subspace $H$ of $\R^n$, one has 
$$
\int_H f(x) dx = ||f||_{\infty} \cdot \vol_m(K_m(f)\cap H). 
$$
Given $f,g \in \mathcal{L}^n$, we define the following  logarithmically concave function $f \star g$ by 
$$
f\star g(x) := \sup_{x=x_2+x_2}f(x_1)g(x_2). 
$$
By selecting $g(x) = f(-x)$, we define the difference function of $f$ by 
$$
\Delta_0 f(x) = \sup_{x=x_1-x_2} f(x_1)f(x_2),
$$
which is an analogue of the difference body in the setting of $log$-concave functions. For $f \in \mathcal{L}^n$ and $1 \leq m \leq n$, define the set
$$
L_m(f) =\{x \in \R^n \colon f(x) \geq ||f||_{\infty} \cdot e^{-m} \}. 
$$

\begin{lemma}\label{t:KM1}
Let $g \colon [0,\infty) \to [0,\infty]$ be a non-decreasing convex function that is not identically zero and fixes the origin. For $m \geq 1$, define the quantity $M = \sup_{t > 0} e^{-g(t)} t^m$, and let $t_0$ be its corresponding unique critical point. Then
$$
\frac{Mt_0}{m+1}\leq \int_0^{\infty} e^{-g(t)}t^m dt < c \frac{Mt_0}{\sqrt{m}},
$$
where $c >0$ is some universal constant.  Moreover, $g(2t_0) \geq m \geq g(t_0)$. 
\end{lemma}

\begin{proof} To handle the left inequality, we note that since $g$ is non-decreasing, it must be the case that 
$$
\int_0^{\infty} e^{-g(t)} t^m dt \geq e^{-\liminf_{t\to t_0^-}g(t)} \int_0^{t_0}t^m dt = \frac{Mt_0}{m+1}.
$$
For right-most inequality, we must consider the function $\varphi(t) = g(t) -m\log t$ whose unique critical point is $t_0$. Differentiation of $\varphi$ yields $\varphi_L'(t_0) \leq 0 \leq \varphi_R'(t_0)$. As a consequence, we conclude both that $g_L'(t_0) \leq \frac{m}{t_0} \leq g_R'(t_0)$ and that $g(t_0) + \frac{m}{t_0}(t-t_0)$ is a supporting line of $g$ at $t_0$. The convexity of $g$ implies that, for every $t>0$, $g(t) \geq g(t_0) + \frac{m}{t_0}(t-t_0)$. Therefore,
\begin{align*}
\int_0^{\infty} e^{-g(t)}t^mdt &\leq e^{m-g(t_0)} \int_0^{\infty} e^{-\frac{m}{t_0}}t^m dt\\
&= e^{-m-g(t_0)} \left(\frac{t_0}{m} \right)^{m+1} \int_0^{\infty} e^{-t}t^m dt\\
=e^{-g(t_0)} t_0^m \frac{e^mm!}{m^m}\frac{t_0}{m} \approx M \frac{t_0}{\sqrt{m}}.
\end{align*}
\indent For any $t< t_0$, we have that $g_R'(t) \leq \frac{m}{t_0}$, and hence $g(t_0) \leq g(0) + \int_0^{t_0} m/t_0 = m$.  Finally, we observe that $g(2t_0) \geq g(t_0) + \frac{m}{t_0}(2t_0-t_0) \geq m$, completing the proof. 
\end{proof}

\begin{lemma}\label{t:KM2}
Given $f \in \mathcal{L}^n$ and $1 \leq m \leq n$, 
$$
K_m(f) \subset L_m(f) \subset c \cdot K_m(f)
$$
for some universal constant $c>0$. 
\end{lemma}

\begin{proof} Fix an arbitrary direction $u \in \s^{n-1}$. We compare $\rho_{K_m(f)}(u)$ and $\rho_{L_m(f)}(u)$. Set $g(ru) := -\log(\bar{f}(r))$, with $\bar{f}= f/||f||_{\infty}$, $M = \sup_{r>0} e^{-g(r)}r^{m-1}$, and let $r_0$ be its corresponding unique critical point. Using Lemma~\ref{t:KM1}, we obtain the estimates 
$$
M \cdot \frac{r_0}{m} \leq \int_0^{\infty} e^{-g(r)}r^{m-1}dr \leq c \cdot \frac{Mr_0}{\sqrt{m-1}},
$$
or equivalently, that 
$$
\rho_{K_m(f)}(u) \approx \left(M \cdot r_0 \right)^{1/m}.
$$
Note that $(M\cdot r_m)^{1/m} \left(r_m^m\right)^{1/m} = r_0$, which implies that $\rho_{K_m(f)}(u) \leq r_0$.  Let $g^{-1}(r) = f^{-1}(||f||_{\infty}e^{-r})$ denote the inverse of $g$. Applying Lemma~\ref{t:KM1}, we see that $g(2r_0) \geq m \geq g(r_0)$, and by applying $g^{-1}$, we see that $r_0 \leq f^{-1}(||f||_{\infty} e^{-m}) \leq 2r_0$, or equivalently, that $f(r_0) \leq ||f||_{\infty}e^{-m} \leq f(2r_0)$, so that $\phi_{L_m(f)}(u) \in [r_0,2r_0]$. This implies precisely the desired inclusions above. 

\end{proof}

\begin{lemma}\label{t:KM3}
Let $f,g \in \mathcal{L}^n$, $1 \leq m \leq n$, and $\theta \in [0,1]$. Then 
 $$
K_m(f \star g) \subset c_0 \cdot (K_m(f) + K_m(g)),
$$
for some universal constant $c_0$. 
\end{lemma}

\begin{proof}[\bf Proof.] Let $\bar{f}= f/||f||_{\infty}$ and $\bar{g}=g/||g||_{\infty}$. Let $x \in L_m(\bar{f} \star \bar{g})$. In view of Lemma~\ref{t:KM2}, we have that $x \in L_m(f \star g)$. Therefore there exist $x_1,x_2 \in \R^n$ such that $x=x_1+x_2$ and $\bar{f}(x_1)\bar{f}(x_2) \leq e^{-m}$. Since $\bar{f}$ and $\bar{g}$ do not exceed one, it follows that $\bar{f}(x_1) \geq e^{-m}$ and $\bar{g}(x_2) \geq e^{-m}$, which implies that $f(x_1) \geq ||f||_{\infty} e^{-m}$ and $g(x_2) \geq ||g||_{\infty} e^{-m}$, and hence $x_1 \in L_m(f)$ and $x_2 \in L_m(g)$, and upon applying Lemma~\ref{t:KM2}, desired inclusion follows.
\end{proof}

\section*{Acknowledgements} The author would like to thanks Artem Zvavitch, Andrea Colesanti, and Sergii Myroshnychenko for many helpful comments that helped to greatly improve the quality of the manuscript. The author would also like to thank the anonymous referee for many help comments and remarks that helped to improve the manuscript.  Finally, the author would to thank the Universidad de Murica and University of Florence for their hospitality in the summer of 2019 where part of this research was established. 

\providecommand{\bysame}{\leavevmode\hbox to3em{\hrulefill}\thinspace}
\providecommand{\MR}{\relax\ifhmode\unskip\space\fi MR }
\providecommand{\MRhref}[2]{%
	\href{http://www.ams.org/mathscinet-getitem?mr=#1}{#2}
}
\providecommand{\href}[2]{#2}

\end{document}